\title{Existence of attracting invariant $2$-curves in fibred quadratic dynamics}
\author{Igsyl Dom\'inguez%\thanks{Supported by ANID grant number 21191154.
        }
\documentclass{article}
\usepackage[a4paper, total={5.5in, 8.2in}]{geometry}
\usepackage{graphicx}
\usepackage{setspace}
\usepackage{graphics}
\usepackage{setspace}
\usepackage{longtable}
\usepackage{epigraph}
\usepackage{indentfirst}
\usepackage{verbatim}
\usepackage{array}
\usepackage{subfigure}
\usepackage{fancyhdr}
\usepackage{mathrsfs}
\usepackage[small,bf]{caption}
\usepackage{cancel}
\usepackage{epstopdf}
\usepackage[utf8]{inputenc}
\usepackage{amsmath,amssymb,latexsym}
\usepackage{amsfonts,amstext,amsthm,amscd}
\usepackage{color}
\usepackage{pst-node}
\usepackage{tikz-cd}
\usepackage{xcolor}
\usepackage{blindtext}
\theoremstyle{plain}
\usepackage[all]{xy}
\usepackage{booktabs}

\newtheorem{lemma}{Lemma}
\newtheorem{proposition}{Proposition}
\newtheorem{corollary}{Corollary}
\newtheorem{definition}{Definition}
\newtheorem{conjecture}{Conjecture}
\newtheorem{remark}{Remark}
\newtheorem{example}{Example}

\renewenvironment{proof}{{\bf \emph{Proof.} }}{\hfill $\Box$ \\} 
\theoremstyle{plain}
\newtheorem*{theorem*}{Theorem}
\newtheorem*{corollary*}{Corollary}

\newcommand{\RR}{\mathbb{R}}
\newcommand{\QQ}{\mathbb{Q}}
\newcommand{\CC}{\mathbb{C}}
\newcommand{\DD}{\mathbb{D}}

\newcommand{\NN}{\mathbb{N}}

\newcommand{\TT}{\mathbb{T}^1}

\newcommand{\Fat}{\mathcal{F}}

\newcommand{\cC}{\mathcal{C}}
\newcommand{\rota}{\mathcal{R}_\alpha}

\begin{document}
\newpage
\maketitle
%%%%%%%%%%%%%%%%%%%%%%%%%%%
% abstract, keywords and Subject classification are optional.
%%%%%%%%%%%%%%%%%%%%%%%%%%%
\begin{abstract}
    We present a construction of new invariant sets for fibred polynomial dynamics with base an irrational rotation over the unit circle, called \emph{multi-curves}. Furthermore, the local dynamical theory for attracting invariant curves is extended to these objects. 
\end{abstract}

% Most people don't use these, so they are "commented out"
% by starting the lines with a "%"
%\begin{keywords}
%   \LaTeX, typesetting
%\end{keywords}

%\begin{AMS}
%   50C60, 18C25
%\end{AMS}

%%%%%%%%%%%%%%%%%%%%%%
% % Here is the start of the Text
%%%%%%%%%%%%%%%%%%%%%%
\section{Introduction}
In a given dynamical system, it is possible to find several types of invariant objects, such as fixed and periodic points, a minimal complex attractor, and the support of an invariant measure, among others. Through these objects, we can understand many features of the system. 

In classic one-dimensional complex dynamics, the Julia set concentrates the most significant (chaotic) part of the dynamics. Nevertheless, we can focus on a simpler invariant set, namely the repelling periodic orbits, since it is a classical result from Fatou and Julia that those objects are dense in the Julia set. 

A continuous map $P:\TT\times\CC\to\TT\times\CC$ is called a \emph{fibred polynomial dynamics} with base an $\alpha$-irrational rotation if $P(\theta,z)=(\rota(\theta),p_\theta(z))$, where $\rota$ is an irrational rotation of $\TT$, and $p_\theta$ is a polynomial for each $\theta$, whose coefficients depend continuously on $\theta$. 

Given that the  irrational rotations of the circle are minimal, it follows that  these fibred polynomials do not contain either fixed or periodic points. It raises the natural question of the existence of minimal invariant objects, distinct from the Julia set for fibred polynomial dynamics. Given the nature of the base space of fibred polynomial dynamics, it is logical to expect that a minimal invariant object possesses the same topological structure of the base space, namely a closed curve.

In his doctoral thesis \cite{PonceTesis}, M. Ponce proved that a natural extension for fixed and periodic points is  \emph{invariant curves}, see Definition \ref{def_inv} below. 
\begin{remark}
    It is important to note that, unlike the classic complex case, the existence of invariant curves for fibred dynamics is a \textit{cohomological problem} rather than an algebraic problem. 
\end{remark}

The example presented in \cite{Mario2} has demonstrated that despite the similarities, there exist significant differences between fibred and classic polynomial dynamics. The aim of this work is to demonstrate another difference in this context by proving the existence of attracting invariant objects that are not simple invariant curves, but have the topological structure of a curve. This new object will be called a \emph{multi-curve}.

This work is organized as follows: In Section 2 we recover the definition of the invariant curve and its main linearization results. In Section 3, we define our objects of study and extend the results of the previous section. In Section 4 we construct a mechanism to obtain 2-invariant attractor curves for quadratic fibered polynomials. Finally, in Section 5, we use the techniques from the previous section to obtain a 3-curve invariant for a bundled rational dynamics.

\section{Invariant curves}\label{sec_inv_curves}
In this section, we describe the concept and basic properties of the primer invariant objects that can be found in \emph{fibred polynomial dynamics} (fpd) with base an irrational rotation. We refer to \cite{DoPo23, Ponce07, Sester} for fundamentals on fibred dynamics.

\begin{definition}\label{def_inv}
    Let $P:\TT\times\CC\to\TT\times\CC$ be a fpd with base an irrational rotation $\rota$. We say that a simple closed continuous curve $\gamma:\TT\to\CC$ is an \textbf{invariant curve} for $P$ if $\gamma$ holds the following condition
    \begin{equation}\label{cond_curve_inv}
        p_\theta(\gamma(\theta))=\gamma(\theta+\alpha), 
    \end{equation}
    for every $\theta\in\TT$. 
\end{definition}
%\subsection{Local dynamics. Multiplier}
Similar to the classic multiplier of complex fixed points, there is a fibred version of this number that locally characterize the dynamics of an invariant curve. 
\begin{definition}
Let $u: \TT \rightarrow \CC$ be an invariant curve for a fibred polynomial $P(\theta,z)=(\rota(\theta),p_\theta(z))$ over an irrational rotation. Provided that the function $\theta \to \log|p_{\theta}'(u(\theta))|$ is a $L^1(\TT)$ function, we call the \textbf{multiplier of the curve} to the positive number 

$$\kappa(u)= \exp\Big(\int_{\TT}\log|p_{\theta}'(u(\theta))|d\theta\Big).$$

When the multiplier $\kappa(u)<1$ we say that the curve is \textbf{attracting}. If $\kappa(u)>1$ we call it \textbf{repulsor}, and in the case $\kappa(u)=1$ the curve is called \textbf{indifferent}. 
\end{definition}

When the invariant curve is also a critical curve, i.e. $p_\theta'(u(\theta))=0$ for every $\theta\in\TT$, it is possible to extend the definition by making $\kappa (u)=0$, note that this is the case for the constant curve $z\equiv \infty$. We call such an invariant curve \textbf{super-attracting}. The integrability condition of the invariant curve allows non-empty intersections between the invariant curves and the critical set, only on finite sets.

In holomorphic dynamics, the multiplier provides us with information about the local dynamics around the respective cycle. For instance, in the attracting case,  we can find a neighborhood of the cycle that is ``\emph{attracted}'' to it. This local dynamical theory has been extended to the fibred case for invariant curves, see \cite{DoPo23,Ponce07} for further references. In particular, in \cite{Ponce07} the author considers local linearization under the extra condition that the fibred polynomial to be injective on the invariant curve, whereas a recent result in \cite{DoPo23} performs an analogous linearization when $\log|p_\theta'(u(\theta))|$ is just a $L^1(\TT)$ function, allowing (finitely many) critical points on the invariant curve. Moreover, the concept of \emph{basin of attraction} for an attracting invariant curve is defined and proved to be an open subset of the fibred space. In the next section, we extend these results for multi-curves.

\section{Multi-curves}\label{multi_curves}
The aim of this section is to describe \textit{multi-curves} as dynamical objects, and to extend the local theory of invariant curves to them. 

Let $\tilde{\gamma}:\TT\to\CC$ be a simple closed curve. For each $n\in\NN$, $\tilde{\gamma}$ induce a simple closed curve in the fibred space $\TT\times\CC$ given by 
\begin{displaymath}
    \begin{array}{rccc}
       \gamma:  & \TT & \to & \TT\times\CC \\
                & \theta & \mapsto & ( \langle n \theta\rangle, \Tilde{\gamma}(\theta)),
    \end{array}
\end{displaymath}
where $\langle\cdot\rangle$ denotes the fractional part. In other words, the image $\Gamma=\gamma(\TT)$ is a closed curve in $\TT\times\CC$ without self-intersections, turning $n$-times in the direction of the base space $\TT$. 

\begin{definition}\label{def_multi_curve}
    We say that a subset $\Gamma\subset\TT\times\CC$ is  a \textbf{$n$-curve} if it is the image of a curve $\gamma$ induced by some $\tilde\gamma:\TT\to\CC$ as described above.  In general, a subset $\Gamma\subset \TT\times\CC$ is called a \textbf{$(p,n)$-curve} or \textbf{multi-curve} if $\Gamma$ consists of $p$ components, each of which is a $n$-curve.
\end{definition}
If we extend $\tilde\gamma$ to $\TT\times\CC$, we may think of $\tilde\gamma$ as a \emph{lifting} of the $n$-curve $\gamma$ under the $n$-fold covering $\Pi_n:\TT\times\CC\to\TT\times\CC$ given by $(\theta,z)\mapsto(\langle n\theta\rangle,z)$. 

Note that if $\Gamma\subset\TT\times\CC$ is a $n$-curve, then the fiber over $\theta$ contains exactly $n$-points for each $\theta\in\TT$, i.e. $|\Gamma_\theta|=n\ \forall\theta$. Hence, if we set a base point $(0,g_0)\in\Gamma$, then there exists a unique $\tilde\gamma$ with $\tilde\gamma(0)=g_0$, in other words, there are $n$ lifts $\tilde\gamma$ for every $n$-curve.

Also, it is not difficult to notice that a $n$-curve $\gamma$  consists of a concatenated list of curves $\gamma_1,...,\gamma_n:[0,1]\to\TT\times\CC$ satisfying $\gamma_i(1)=\gamma_{i+1}(0)$, with $\gamma_{n+1}=\gamma_1$. Once we set a base point $(\theta_0,g_0)=(0,\tilde\gamma(0))$ in $\Gamma$, each curve may be defined as
\begin{equation}\label{eq_unf1}
    \gamma_i(\theta) = {\gamma}\left(\dfrac{i-1 + \theta}{n}\right), \ i=1,2,...,n.
\end{equation}

In this sense, we denote a $n$-curve $\Gamma$ with base point $(0,\tilde\gamma(0))$ as $\Gamma=(\gamma_1\ \gamma_2\ ...\ \gamma_n)$.

\begin{figure}[ht]
\includegraphics[width=8cm]{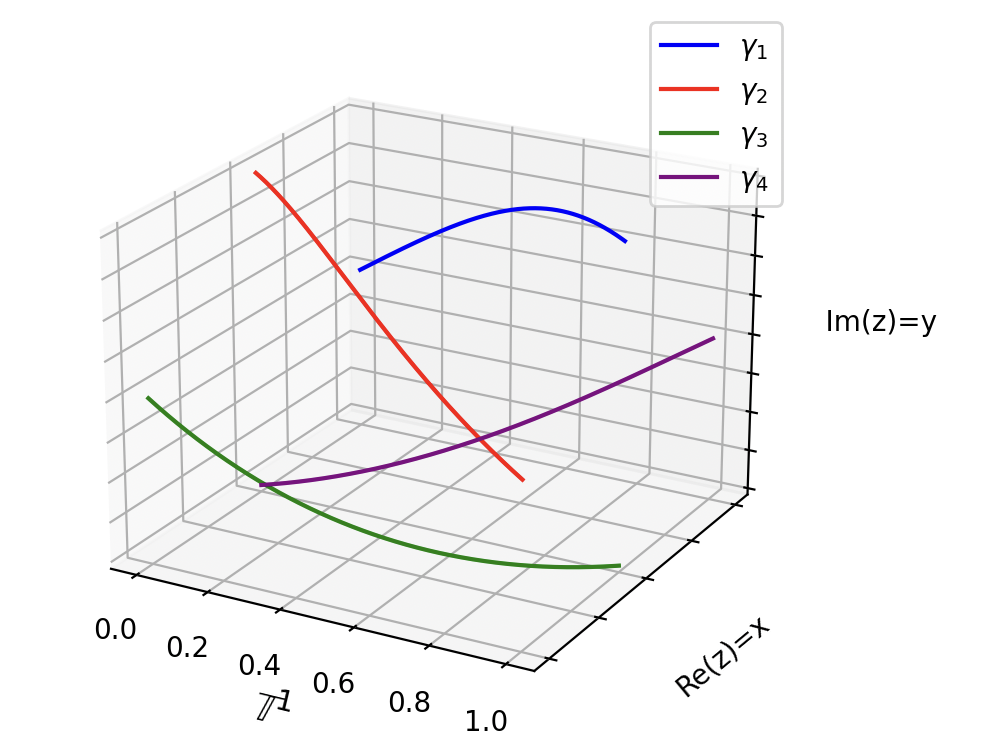}
\centering
\caption{Displaying a $4$-curve in $\TT\times\CC$}
\end{figure}

\begin{figure}[ht]
\includegraphics[width=8cm]{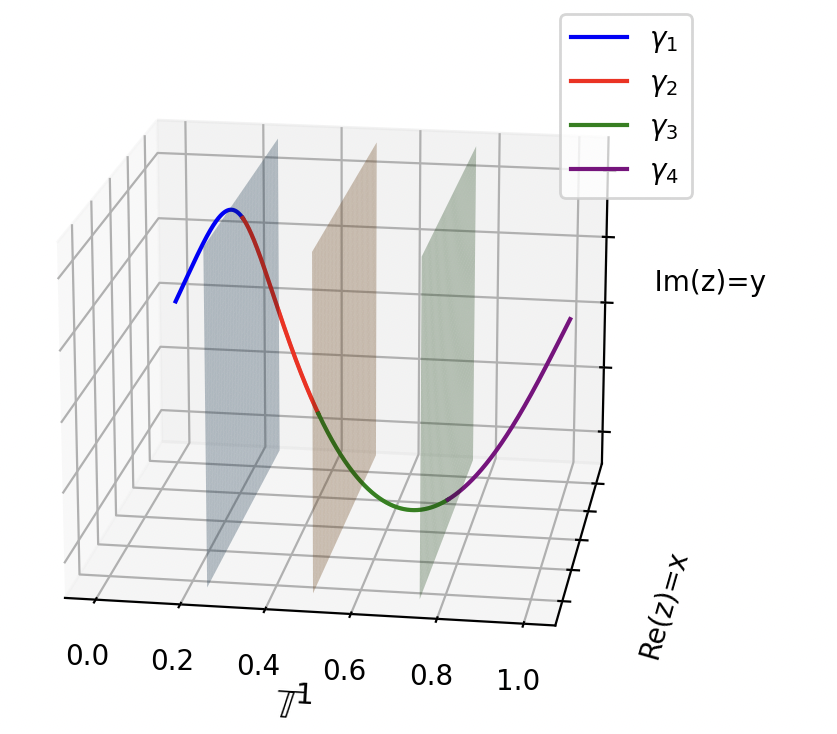}
\centering
\caption{``Unfolding'' process for the $4$-curve}
\end{figure}

%This new extended curve may be thought as an \emph{`unfolding'} of the $n$-curve $\gamma$. We can reparametrized the curve $\Tilde{\gamma}$ to be defined in the unit interval, and hence the base space $\TT$.
\begin{definition}\label{def_unf}
    Let $\Gamma=(\gamma_1\ \gamma_1\ ...\ \gamma_{n})$ be a $n$-curve induced by an injective continuous function $\Tilde{\gamma}:\TT\to\CC$, the extended image of $\tilde\gamma$ in the fibred space $\TT\times\CC$, as given by 
    \begin{equation}\label{eq_unf2}
    \tilde{\gamma}(\theta) = (\theta,\tilde\gamma(\theta)),
    \end{equation}
     is called the \textbf{unfolding curve} of $\gamma$.  
\end{definition}

\begin{remark}
    Note that the unfolding $\tilde\gamma$ as above defined is actually the \emph{unique} lift of the $n$-curve $\gamma:\TT\to\TT\times\CC$ with base point $(0,\tilde\gamma(0))$. 
\end{remark}
\begin{figure}[ht]
\includegraphics[width=10cm]{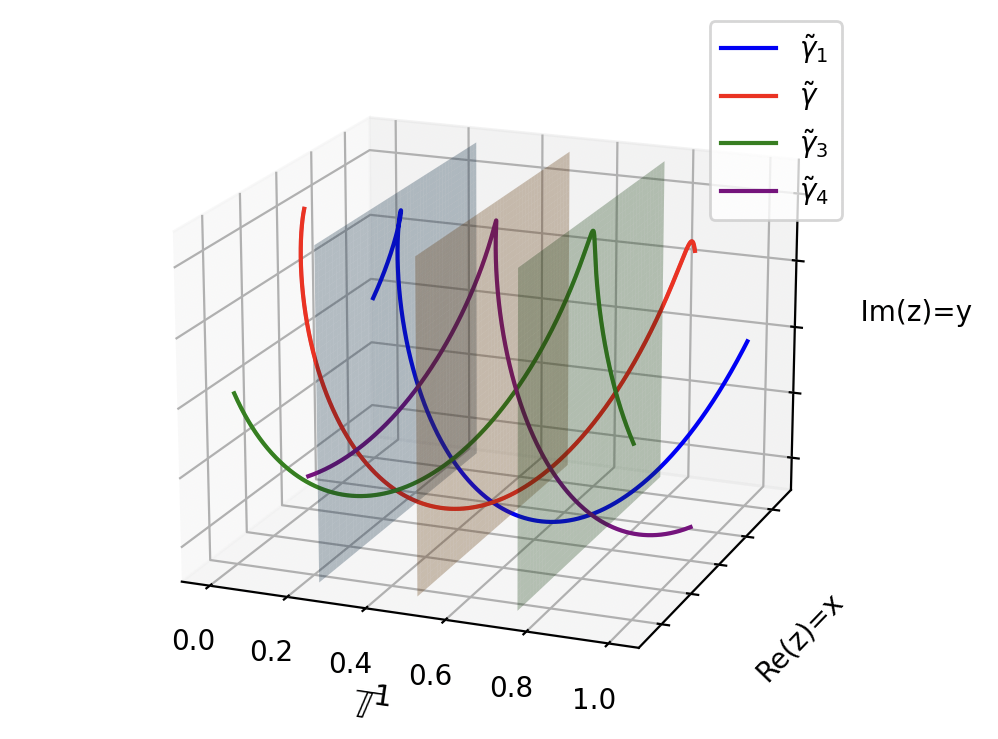}
\centering
\caption{4 \emph{unfoldings} for a $4$-curve}
\end{figure}

\subsection{Invariant multi-curves} 

Consider now a fibred holomorphic dynamics $F:\TT\times\CC\to\TT\times\CC$ and let $\Gamma = (\gamma_1\ \gamma_2\ ...\ \gamma_{n})$ be a $n$-curve in $\TT\times\CC$ such that $F\big|_{\Gamma}$ is a homeomorphism, i.e. $\Gamma= (\gamma_1\ \gamma_2\ ...\ \gamma_{n})$ is \emph{invariant} under $F$ as a subset. For each $\theta\in\TT$, the fiber of $\Gamma$ over $\theta$ consists of $n$ distinct points, and then Equation (\ref{def_inv}) makes no sense as an invariant notion. 
Although $\Gamma=(\gamma_1\ \gamma_2\ ...\ \gamma_{n})$ is invariant under $F$ as a subset of $\TT\times\CC$, the orbit of $\Gamma$ may be `dynamically jumping' along the concatenated list of curves $(\gamma_1\ \gamma_2\ ...\ \gamma_{n})$.

Consider the $n$-fold covering of $\TT\times\CC$ given by 
\begin{displaymath}
    \begin{array}{rccc}
        \Pi_n: & \TT\times\CC & \to & \TT\times\CC \\
             & (\theta,z) & \mapsto & (\langle n\theta\rangle,z),
    \end{array}
\end{displaymath}
then, every \emph{lifting} of $F$ under $\Pi_n$ is given by 

\begin{displaymath}
    \begin{array}{rccc}
       \tilde{F}_\tau:  & \TT\times\CC & \to & \TT\times\CC \\
                 & (\theta, z) & \mapsto & \left(\theta + \dfrac{\alpha+\tau}{n}, f_{\langle n\theta\rangle}(z)\right),
    \end{array}
\end{displaymath}
for some $\tau\in\{0,1,...,n-1\}$. We recall that $T(\theta)=\theta+\tau/n$ is a \textbf{deck transformation} for $\Pi_n$. 
We obtain the following commutative diagram.
\begin{equation}\label{comm_diag}
\xymatrix{
  \TT\times\CC \ar[d]_{\Pi_n} \ar[r]^{\tilde{F}}  &
             \TT\times\CC \ar[d]^{\Pi_n} \\
  \TT\times\CC  \ar[r]_{F}  & \TT\times\CC . }
\end{equation}
%Notice that if we divide the interval $[0,1]$ into $n$ sub-intervals of length $1/n$, the integer $\tau$ in the fibred dynamics $\Hat{F}_\tau$, represents a jump among these intervals. We then consider the following definition. 

\begin{definition}\label{def_mult-curve}
    Suppose $\Gamma=(\gamma_1\ \gamma_2\ ...\ \gamma_{n})$ is a $n$-curve (or multi-curve) induced by a curve $\tilde\gamma:\TT\to\CC$. We say that $\Gamma$ is a \textbf{dynamically invariant curve} (or \textbf{invariant multi-curve}) for the fibred dynamics $F:\TT\times\CC\to\TT\times\CC$ if the curve $\tilde\gamma:\TT\to\CC$ is invariant for some lifting $\tilde{F}_\tau:\TT\times\CC\to\TT\times\CC$ as defined in (\ref{def_inv}). 
\end{definition}
The above commutative diagram makes this definition well-defined. More over, it is clear that $\Hat{F}_\tau\big|_{\Hat{\Gamma}}$ is a homeomorphism if and only if $F\big|_{\Gamma}$ is a homeomorphism. Finally, combining the invariance of $\tilde\gamma$ and the commutative diagram, it follows that $\tau$ determine how $F\big|_{\Gamma}$ dynamically jumps among the curves $(\gamma_0\ \gamma_1\ ...\ \gamma_{n-1})$, that is, for every $\theta\in\TT$ we have
\begin{equation}\label{eq_jump}
    f_\theta(\gamma_i(\theta)) = \gamma_{i+\tau}(\theta+\alpha).
\end{equation}
This way, $\tau$ is called the \textbf{jumping integer} for $\Gamma = (\gamma_1\ \gamma_2\ ...\ \gamma_{n})$. 

\subsection{Dynamically invariant multi-curves exist.}
In this short subsection, we exhibit a couple of examples of multi-curves for fibred dynamics. The examples are extreme opposite in the sense that the former is a trivial construction of several multi-curves for a fibred dynamics in the unit circle (with rotation as base map), while the further is a forced construction of a fibred polynomial dynamics based on a given \emph{topological multi-curve}. 
\begin{example}\label{ex_triv} \emph{
    Let $n\in\NN$, $\alpha\in\TT$, and $F:\TT\times\TT\to\TT\times\TT$ be a fibred dynamics on the unit circle over a rotation by $\alpha$ in the unit circle itself, $F(x,y)=(x+\alpha,y+\dfrac{\alpha}{n})$. Then the curve $\gamma:[0,n]\to\TT$ defined by 
    $$
    t\mapsto\gamma(t)=\dfrac{t}{n},
    $$
    induce an invariant $n$-curve for $F$. In fact, the phase space $\TT\times\TT$ is foliated by invariant copies of this invariant $n$-curve. }
\end{example}
\begin{example}\label{ex_lag} \emph{
The Interpolation Lagrange polynomial is a very useful tool to construct \emph{invariant multi-curves}. Let $\Gamma = (\gamma_1\ \gamma_2\ ...\ \gamma_{n})$ be a $n$-curve, $\alpha\in\RR$, and $\tau\in \{0, 1,...,n-1\}$. For every $\theta\in\TT$, let $p_\theta$ be the $n-1$ degree Lagrange interpolation polynomial taking the points $\{\gamma_1(\theta), \gamma_2(\theta),...,\gamma_{n}(\theta)\}$ to the points $\{\gamma_{1+\tau}(\theta+\alpha), \gamma_{2+\tau}(\theta+\alpha),...,\gamma_{n-1+\tau}(\theta+\alpha)\}$ sending point $\gamma_{i}(\theta)$ to the point $\gamma_{i+\tau}(\theta+\alpha)$, where $i+\tau$ is taken $(\mod n)$. Then, the fibred polynomial
\begin{displaymath}
    \begin{array}{rccl}
        P: & \TT\times\CC & \to & \TT\times\CC \\
           & (\theta,z) & \mapsto & (\theta + \alpha, p_\theta(z))
    \end{array}
\end{displaymath}
is continuous and leaves $\Gamma$ (dynamically) invariant with jumping integer equal to $\tau$. A similar construction can be made to get a fibred higher degree polynomial dynamics ($(n+p-1)$ degree) leaving invariant a prescribed $(p,n)$-curve.
}
\end{example} 
\subsection{Dynamical nature of multi-curves} 
One wonders if it is possible to determine a (locally) dynamical nature of an invariant multi-curves as for simple invariant curves. This will be possible since the invariance of the multi-curve is defined through a simple (unfolding) invariant curve. 

The fibred multiplier can then be extended for invariant multi-curves in the following way. 
\begin{definition}
    Suppose that $\Gamma=(\gamma_1\ \gamma_2\ ...\ \gamma_n)$ is a (dynamically) invariant $n$-curve for the fibred polynomial $P(\theta,z)=(\theta+\alpha,p_\theta(\theta))$, then the \textbf{fibred multiplier} of $\Gamma$ is defined as 
    \begin{equation}\label{multi-multiplier}
        \kappa_f(\Gamma):=\kappa(\tilde{\gamma})=\exp\left(\int_{\TT}\log |\partial_z\hat{p}_\theta(\tilde{\gamma}(\theta))|d\theta\right),
    \end{equation}
    where $\hat{P}(\theta,z)=\left(\theta + \dfrac{\tau+\alpha}{n}, \hat{p}_\theta(z)\right)$ and $\tilde{\gamma}$ are the lifted \emph{fpd} and its (unfolding) invariant curve associated. Moreover, the multi-curve $\Gamma$ is called \textbf{attracting}, \textbf{repulsor} or indifferent if $\kappa<1$, $\kappa>1$ or $\kappa=1$ respectively.
\end{definition}
\begin{remark}
     In Example \ref{ex_lag}, by increasing the degree of $p_\theta$, we can impose extra mild conditions on the complex derivative $\partial_z P$ at points of the multi-curve $\Gamma$, so that $\Gamma$ yields into an attracting invariant $n$-curve. 
\end{remark}
Defining the multiplier of a multi-curve through its unfolding curve, allows us to extend the local theory for multi-curves from Section \ref{sec_inv_curves}. The following results are direct consequences of the ones in \cite{DominguezThesis, DoPo23, Ponce07} through the commutative diagram (\ref{comm_diag}). 

\begin{lemma}[The attracting case]
    Let $P$ be a fibred polynomial dynamics over an irrational rotation, and $\Gamma$ be an attracting invariant multi-curve. Then there exists a continuous change of coordinates $H(\theta,z)=(\theta,a(\theta)z+b(\theta))$ such that $\Gamma$ is still an attracting invariant multi-curve for the conjugated fibred polynomial dynamics $Q=H^{-1}\circ P\circ H$. Moreover, if $Q(\theta,z)=(\theta+\alpha, q_\theta(z))$, then there exists $c<1$ such that
    $$
    \sup_{\theta\in\TT}|\partial_z q_\theta(\gamma(\theta))|<c.
    $$
\end{lemma}

\begin{lemma}[The repulsor case]
    Let $P$ be a fibred polynomial dynamics over an irrational rotation, and $\Gamma$ be a repulsor invariant multi-curve. Then there exists a continuous change of coordinates $H(\theta,z)=(\theta,a(\theta)z+b(\theta))$ such that $\Gamma$ is still a repulsor invariant multi-curve for the conjugated fibred polynomial dynamics $Q=H^{-1}\circ P\circ H$. Moreover, if $Q(\theta,z)=(\theta+\alpha, q_\theta(z))$, then there exists $c>1$ such that
    $$
    \inf_{\theta\in\TT}|\partial_z q_\theta(\gamma(\theta))|>c.
    $$
\end{lemma}
Also, the basin of attraction is well-defined in the attracting case. 

\begin{lemma}\label{multi_tub}
    Let $P$ be a fibred polynomial dynamics over an irrational rotation, and let $\Gamma$ be an attracting invariant \emph{n}-curve. Then there exists an open set $\mathcal{T}\subset \TT\times\CC$ containing the multi-curve $\Gamma$, and such that every point in $\mathcal{T}$ is attracted to $\Gamma$, i.e., $z_\theta\in\mathcal{T}_\theta$ then 
    $$
    \emph{dist}(P^n(\theta, z_\theta), \Gamma)\to 0,\ \emph{as}\ n\to\infty.
    $$
    Moreover, for every $\theta\in\TT$, the fiber $\mathcal{T}_\theta$ consists of \emph{n}-components each of which contains a point of $\Gamma_\theta$.
\end{lemma}
The open set $\mathcal{T}$ defined in the above lemma, may be thought as a (neighborhood) \textbf{multi-tube} around the multi-curve $\Gamma$. This allows us to formally define its basin of attraction. 

\begin{definition}
    Let $P$ be a fibred polynomial dynamics over an irrational rotation, and let $\Gamma$ be an attracting invariant multi-curve. If $\mathcal{O}^+(\theta,z)$ defines the forward orbit under $P$ of a point $(\theta,z)$, then 
    $$
    \mathcal{A}(\Gamma)=\{(\theta,z):\emph{dist}(\mathcal{O}^+(\theta,z),\Gamma)\to0,\ n\to\infty\}
    $$
    is called \textbf{the basin of attraction of the multi-curve} $\gamma$.
\end{definition}

Analogous to the simply invariant case, we have that 
$$
\mathcal{A}(\Gamma)=\bigcup_{n\geq0}P^{-n}(\mathcal{T}),
$$
where $\mathcal{T}$ is the invariant multi-tube defined in Lemma \ref{multi_tub}.
\begin{corollary}
    If $\Gamma$ is an attracting invariant multi-curve, then its basin of attraction $\mathcal{A}(\Gamma)$, is an open subset of $\TT\times\CC$. 
\end{corollary}
\section{Invariant 2-curves for small perturbation of a static quadratic dynamics}

In this section multi-curves are exhibited in the lowest grade for polynomials where interesting dynamics appear: the quadratic case (recall trivial Example \ref{ex_triv} is of degree one). 
\subsection{Fibred quadratic polynomials}
Consider the family $\Fat_\alpha$ of \textbf{canonical} fibred quadratic polynomials 

\begin{equation}\label{gen_qua}
    \begin{array}{rccl}
        P^\alpha_\mathcal{C}: & \TT\times\CC & \to & \TT\times\CC \\
           & (\theta,z) & \mapsto & (\theta + \alpha, z^2+\mathcal{C}(\theta)),
    \end{array}
\end{equation}
where $\alpha\in\TT$ is the irrational rotation angle and $\mathcal{C}:\TT\to\CC$ is a continuous function that may be thought as a parameter. A wider family of quadratic polynomial dynamics has been widely studied by Sester in \cite{Sester}, here the author defines the corresponding \emph{principal cardioid of the fibred Mandelbrot set}. We will be only focusing on those quadratic polynomials with a good normalization, as given in \ref{gen_qua}.

The idea in the construction will be to find invariant $n$-curves by choosing a parameter $\mathcal{C}:\TT\to\CC$ that \emph{wanders} through some special places around the classical Mandelbrot set. 

Analogous to the classic one-dimensional complex case, under a mild condition on the quadratic coefficient, every quadratic polynomial can be normalized to the form in (\ref{gen_qua}) under a suitable fibred change of coordinates, see Proposition 2.1 and 2.2 in \cite{Sester} for further reference on this.

\subsubsection{The `static' fibred polynomial}

Let $P(\theta,z) = (R_\alpha(\theta), p_\theta(z))$ be a fibred polynomial dynamics, where $p_\theta(z)$ is a degree $d$ polynomial and $\alpha=0$, in other words, $P(\theta,z)$ may be viewed as a continuous parametrized family of polynomial dynamics. We refer to this case as the \emph{static fibred polynomial} case. 

Suppose that $\mathcal{Z}\subset \TT\times\CC$ is a connected component of the continuous solution to the \emph{fixed-points} equation: 
\[p_\theta(z(\theta)) = z(\theta).\]
Since $\alpha=0$, it follows that $P\big|_{\mathcal{Z}}$ is a homeomorphism. Suppose also that $\mathcal{Z}$ is a $n$-curve, then it is easy to see that $\mathcal{Z}$ is an \emph{invariant multi-curve} according to Definition \ref{def_mult-curve}. 

The strategy of this chapter is to construct invariant multi-curves through the curves generated by the set of fixed points of the static fibred polynomial. We will consider suitable parametric curves (small circle with the classical parabolic parameter $c_0=1/4$ in its interior). Then, after a post-composition with a Lagrange Interpolation Polynomial, adding the fibred nature with the irrational rotation as described in Example \ref{ex_lag}, we will maintain the invariance of the multi-curve for a fibred polynomial dynamics which will be still quadratic since the Lagrange Interpolation polynomial will be linear.  
\subsection{Fixed Points of the quadratic polynomial}
Consider the {canonical} form of a quadratic polynomial
\[q_c(z) = z^2 + c, \ z\in\CC^*,\]
it is well known that the fixed points of $q_c$ are given by 
\begin{equation}\label{static_fixed}
  z_1(c) = \frac{1}{2} + \sqrt{\frac{1}{4} - c}, \qquad \& \qquad z_2(c) = \frac{1}{2} - \sqrt{\frac{1}{4} - c}  
\end{equation}
If $c=1/4$, $q_c$ possesses one, and only one fixed point; otherwise, there are always two distinct fixed points. It is a classical well-known fact the $c=1/4$ is a parabolic parameter of the Mandelbrot set, whose Julia set has the form of a Cauliflower. 

Let $\varepsilon >0$, $x_0>0$ such that

\begin{equation}\label{4.a}
0\leq x_0<1/4\ \ \ \text{and}\ \ \ \ \frac{\varepsilon^2}{\frac{1}{4}-\varepsilon^2x_0}<1
\end{equation}

and consider the parametric curve $\cC:\TT\rightarrow \CC$ given by,

\begin{equation}\label{model_sta}
    \cC(\theta)=\frac{1}{4}-\varepsilon^2 x_0-\varepsilon^2 e^{2\pi i \theta},
\end{equation}
that is, $\cC:\TT\to \CC$ is a simple continuous loop around the parabolic parameter $c=1/4$. If we keep tracking the fixed points $z_1(\cC(\theta)), z_2(\cC(\theta))$ when $\theta$ goes from $0$ to $1$ (in $\TT$), we see that this tour to the loop of parameters gives rise to a transposition of the fixed points, as we will see in the next lines.

Now, if we substitute the form of $c=\cC(\theta)$ in the solutions (\ref{static_fixed}) we have 
$$z_1(\theta)=\frac{1}{2}+ \epsilon e^{\pi i \theta}\sqrt{1+x_0e^{-2\pi i \theta}}$$
and
$$z_2(\theta)= \frac{1}{2}-\epsilon e^{\pi i \theta}\sqrt{1+x_0e^{-2\pi i \theta}},$$
and they form a set of fixed points, and if we take $-1 = e^{\pi i}$ their concatenation $\tilde\gamma : \TT\rightarrow \CC$
    $$\tilde\gamma(\theta) = 1/2 + \varepsilon e^{2\pi i\theta}\sqrt{1 + x_0e^{-4\pi i\theta}},$$
induces an invariant 2-curve $\Gamma=(z_1 \ z_2)$ for the static quadratic polynomial $P(\theta,z)$. 

\begin{remark}
As was mentioned above, this $\Gamma$ is our candidate for an invariant $2$-curve for a fibred quadratic dynamics.     
\end{remark}

Consider the static fibred quadratic dynamics, that is, with $\alpha=0$,
\begin{equation}\label{static_pol}
\begin{array}{rccc}
Q: & \TT\times\CC & \to & \TT\times\CC \\
   & (\theta,z) & \mapsto & (\theta, q_{\cC(\theta)}(z)).
\end{array}
\end{equation}.
\begin{proposition}
    The $2$-curve $\Gamma = (z_1\ z_2)$ is invariant for the static fibred polynomial $Q(\theta,z) = (\theta, q_{\cC(\theta)}(z))$.
\end{proposition}
\begin{proof}
Note that the lifting $\widehat Q$ of $Q$, is given by 
\begin{displaymath}
\begin{array}{rccc}
\widehat{Q}: & \TT\times\CC & \to & \TT\times\CC \\
   & (\theta,z) & \mapsto & (\theta, q_{\cC(\langle2\theta\rangle)}(z)),
\end{array}
\end{displaymath} 
with $\Pi_2(\theta,z)=(\langle 2\theta\rangle,z)$. A direct calculation shows that $q_{\cC(\langle2\theta\rangle)}(\tilde\gamma(\theta))=\tilde\gamma(\theta)$, and hence $\Gamma$ is invariant for $Q$ (here $\alpha=0$). 
\end{proof}

Note that if we permute the order of the fixed points $z_1$ and $z_2$, for $\Gamma=(z_2 \ z_1)$, we have another lifting curve $\tilde\gamma_2$ given by the relation
$$
\tilde\gamma_2(\theta) = \tilde\gamma(\theta + 1/2).
$$
We re-label both permutations by $\tilde\gamma_1$ and $\tilde\gamma_2$ respectively. It is not difficult to see that $\tilde\gamma_2$ is also invariant for $\widetilde Q$. 

\subsection{From Static to Fibred. The Post-Composition}\label{sec_multi_inv}
For $\alpha>0$ small enough, we want to ``transform'' $\widehat{Q}$ in such a way that $z_1$ and $z_2$ still form an invariant $2$-curve for a fibred quadratic polynomial. For each $\theta\in\TT$, consider the two pairs of points $(\tilde{\gamma}_1(\theta),\tilde{\gamma}_2(\theta))$ and $(\tilde{\gamma}_1(\theta + \frac{\alpha}{2}), \tilde{\gamma}_2(\theta + \frac{\alpha}{2}))$, and define the linear fibred map 
\begin{displaymath}
\begin{array}{cccc}
\Tilde{L}_{\alpha}: & \TT\times\CC & \to & \TT\times\CC\\
                    & (\theta,\zeta) & \mapsto & (\theta, \Tilde{\ell}_{\theta}(\zeta)),
\end{array}
\end{displaymath}
where $\Tilde{\ell}_{\theta}$ is given by the Lagrange interpolation polynomial (affine, since $n=2$) between the pairs of points considered above. Then, 
$$ \Tilde{\ell}_{\theta}(\tilde{\gamma}_{1}(\theta)) = \tilde{\gamma}_{1}(\theta+\frac{\alpha}{2}), \qquad \text{and} \qquad \Tilde{\ell}_{\theta}(\tilde{\gamma}_{2}(\theta)) = \tilde{\gamma}_{2}(\theta + \frac{\alpha}{2}). $$

The following result is immediate from the construction. 
\begin{proposition}\label{alpha_pos}
For $\alpha>0$, and $\Tilde{\ell}_{\theta}$ and $q_{\cC(\theta)}$ as above, define the fibred quadratic polynomial 
$$ \widetilde{Q} = \widetilde{L}_{\alpha}\circ \widehat{Q}, $$
that is,
\begin{displaymath}
\begin{array}{cccc}
\widetilde{Q}: & \TT\times\CC  & \to & \TT\times\CC \\
           & (\theta, \zeta) & \mapsto & (\theta + \frac{\alpha}{2}, \Tilde{\ell}_{\theta}\circ q_{\cC(\langle 2\theta\rangle)}(\zeta)). 
\end{array}
\end{displaymath}
Then the \emph{unfolding} curve $\tilde{\gamma}$ is an invariant curve for $\widetilde{Q}$.
\end{proposition}

\begin{proposition}
Let $\alpha \in (0,1)\backslash \mathbb{Q}$ small enough. For $\varepsilon$ and $x_0$ as in (\ref{4.a}), there exists a fibred quadratic polynomial in canonical form
$$(\theta,z)\mapsto (\theta + \alpha, z^2+\mathcal{C}_0(\theta))$$
containing an attracting invariant 2-curve. If $x_0=0$, the invariant 2-curve is indifferent. 
\end{proposition}

\begin{proof}
We prove this using the \emph{small perturbation} procedure given by the Lagrange Interpolation Polynomial post-composition, as described above.  
We define our perturbation on the static model,
$$P(\theta, z)= (\theta, z^2+\cC(\theta)),$$
where $\cC(\theta)$ is given by (\ref{model_sta}). We know that the curves
$$z_1(\theta)=\frac{1}{2}+ \epsilon e^{\pi i \theta}\sqrt{1+x_0e^{-2\pi i \theta}}$$
and
$$z_2(\theta)= \frac{1}{2}-\epsilon e^{\pi i \theta}\sqrt{1+x_0e^{-2\pi i \theta}},$$
form a set of fixed points, and their concatenation $\tilde\gamma : \TT\rightarrow \CC$
    $$\tilde\gamma(\theta) = 1/2 + \varepsilon e^{2\pi i\theta}\sqrt{1 + x_0e^{-4\pi i\theta}},$$
induces an invariant 2-curve $\Gamma=(z_1 \ z_2)$ for the static quadratic polynomial $P(\theta,z)$. 

Now, we calculate the multiplier of the (unfolding) invariant curve of the lifted system $\widetilde{P}(\theta,z)$.

For this, note that, if $\Tilde{\gamma}(\theta)$ is the (unfolding) invariant curve for the lifted fibred polynomial, then $\Tilde{\gamma}(\theta)\big|_{[0,1/2]}=z_1(2\theta)$ and $\Tilde{\gamma}(\theta)\big|_{[1/2,1]} = z_2(2\theta)$, and hence 

\begin{align*}
    \log (\kappa(\Tilde{\gamma}(\theta))) & = \int_{\TT} \log|2\Tilde{\gamma}(\theta)|d\theta = \int_0^{1/2}\log|2\Tilde{\gamma}(\theta)|d\theta + \int_{1/2}^1\log |2\Tilde{\gamma}(\theta)|d\theta \\ 
        %& = \int_0^{1/2}\log|2\Tilde{\gamma}(\theta)|d\theta + \int_{1/2}^1\log |2\Tilde{\gamma}(\theta)|d\theta \\
        & = \int_0^{1/2}\log |2z_1(2\theta)|d\theta + \int_{1/2}^1\log |2z_2(2\theta)|d\theta = \int_0^{1/2}\log|4z_1(2\theta)z_2(2\theta)|d\theta \\
        %& = \int_0^{1/2}\log|4z_1(2\theta)z_2(2\theta)|d\theta \\
        & = \int_0^{1/2}\log|4C(2\theta)|d\theta = \dfrac{1}{2}\int_0^1\log|4C(\theta)|d\theta.
\end{align*}
So, for the case $\alpha=0$ it is enough to calculate the integral 
\[
2\log (\kappa(\tilde\gamma)) = \int_0^1\log|4C(\theta)|d\theta. 
\]
From the form in \ref{model_sta} and the conditions on $\varepsilon$ and $x_0$, we have
\begin{align*}
    2\log(\kappa(\tilde\gamma)) & = \int_0^1\log |4C(\theta)|d\theta \\ 
        %& = \int_0^1\log|4(1/4 - \varepsilon^2e^{2\pi i\theta}-x_0)|d\theta \\
        %& = \int_0^1\log |4(1/4-x_0)(1 - \dfrac{\varepsilon^2}{1/4-x_0}e^{2\pi i\theta})|d\theta \\
        & =\int_0^1\log|4(1/4-\varepsilon^2x_0)|d\theta + \int_0^1\log\left|1-\dfrac{\varepsilon^2}{1/4-\varepsilon^2x_0}e^{2\pi i\theta}\right|d\theta,
\end{align*}
so, if 
\[
0\leq x_0<1/4\ \ \ \text{and}\ \ \ \ \frac{\varepsilon^2}{\frac{1}{4}-\varepsilon^2x_0}<1
\]
then $\kappa(\gamma) < 1,$ and if $x_0=0$ $\kappa(\gamma)=1$.

Hence, the curve is attracting (indifferent for $x_0=0$) for the static fibred polynomial. But we are interested in the non-static fibred case $\alpha>0$, that is, after applying Lagrange interpolation and normalizing to the canonical form. 

A direct calculus shows that the fibred quadratic polynomial, in its canonical form obtained this way, is given by 
$$
Q(\theta,z) = \left(\theta+\alpha, z^2+\dfrac{C(\theta)}{u(\theta+\alpha)}\right),
$$
where
\[u(\theta) = [A(\theta)]^{-1}\cdot\prod_{j=0}^\infty\left[\frac{A(\theta+j\alpha)}{A(\theta+(j+1)\alpha)}\right]^{\dfrac{1}{2^{j+1}}},\]
and 
$$
A(\theta) = \dfrac{e^{\pi i\alpha}\sqrt{1+x_0e^{-2\pi i(\theta+\alpha)}}}{\sqrt{1+x_0e^{-2\pi i\theta}}},
$$
with the invariant 2-curve given by 
$$
\Gamma_0 = \dfrac{1}{u(\theta)}\cdot(z_1\ z_2),
$$
and corresponding \emph{unfolding}
$$
\tilde{\gamma}_0(\theta) = \dfrac{\tilde\gamma(\theta)}{u(\theta)}. 
$$
\begin{remark}
Note that if $\alpha>0$ is sufficiently small, the product in the above relation is very close to 1, so $u(\theta)$ is very close to $ [A(\theta)]^{-1}$.    
\end{remark}

So, the multiplier of the (unfolding) invariant 2-curve is 
\begin{align*}
    \kappa(\tilde{\gamma}_0) & = \exp\left(\int\log\left|\dfrac{2\tilde\gamma(\theta)}{u(\theta)}\right|d\theta\right) = \exp\left(\int\log|\tilde\gamma(\theta)|d\theta-\int\log\left|\dfrac{u(\theta)}{2}\right|d\theta\right).
\end{align*}
Hence, for $\tilde{\gamma}_0$ (and then $\Gamma$) to be attracting, it is enough that 
\begin{equation}\label{4.c}
\int\log\left|\dfrac{u(\theta)}{2}\right|d\theta = \int\log\left|2{A(\theta)}\right|d\theta = 0.
\end{equation}
But, $A(\theta)=\dfrac{e^{\pi i\alpha}\sqrt{1+x_0e^{-2\pi i(\theta+\alpha)}}}{\sqrt{1+xe^{-2\pi i\theta}}}$, and from Remark \ref{4.c} Equation (\ref{4.c}) reduces to 
$$
\int\log|\sqrt{1+x_0e^{-2\pi i\theta}}|d\theta = \int\log|\sqrt{1+x_0e^{-2\pi i(\theta+\alpha)}}|d\theta=0,
$$
which follows by noticing that each integral above is the real part of the integral 
$$
\int_{|z|=r}f(z)dz, 
$$
where $f(z)=\sqrt{1+z}$ and $r=x_0$ in the former and $f(z)=\sqrt{1+e^{\pi i\alpha}z}$ and $r=x_0$ in the former. We conclude that 
$$\kappa_f({\Gamma_0})=\kappa(\tilde\gamma_0)<1.$$
This way, the 2-curve $\Gamma_0 = \dfrac{1}{u(\theta)}\cdot(z_1\ z_2)$ is an \emph{attracting invariant} multi-curve for the fibred quadratic polynomial 
$$
P(\theta,z) = \left(\theta+\alpha, z^2+\mathcal{C}(\theta)\right),
$$
with $\mathcal{C}(\theta)=C(\theta)/u(\theta+\alpha)$. It is clear from the proof, that if $x_0=0$, $\kappa(\Gamma_0)=1$, and the 2-curve is indifferent. 
\end{proof}

\subsection*{Fibred combinatorics ($\tau=1$)}
In the above construction, by obtaining the invariance of the curve in the lifting, dynamically we stay over the same `part' of the $2$-curve. We recall that a $n$-curve may have defined a combinatorics ``over the fibred''. \\
For the case of the $2$-curve, there are only two possible combinatorics. 
\begin{itemize}
    \item The dynamics stay in the same part of the curve ($\tau=0$).
    \item The dynamics do ``jumps'' between the two parts ($\tau=1$).
\end{itemize}
It is clear that $\tau=0$ in the above construction. For $\tau=1$, the Lagrange interpolation polynomial (affine) $\tilde{\ell}$, may be defined by the pairs 
$$ (\tilde{\gamma}_0(\theta),\tilde{\gamma}_1(\theta)) \qquad \text{and} \qquad (\tilde{\gamma}_1(\theta + \frac{\alpha}{2}), \tilde{\gamma}_0(\theta + \frac{\alpha}{2})). $$ 

However, the idea of taking $\alpha > 0$ sufficiently small is that the Lagrangian interpolation polynomial $\tilde{l}$ is very close to the identity so that the composed fibered polynomial $P\circ \tilde{l}$ is, in fact, a small perturbation of the static dynamics.

There is another way to obtain a 2-invariant curve with ``jumping integer'' $\tau =1$. Let's consider the parameterized curve, $\varepsilon > 0$ small,

$$C(\theta)=-\frac{3}{4} -\varepsilon^2e^{2\pi i\theta}$$

That is, $C:\TT \rightarrow \CC$ is a simple closed curve around the parameter, which is a parameter with parabolic multiplicity equal to 2.

Given the static quadratic polynomial
$$P(\theta,z)=(\theta, z^2 + C(\theta)),$$

we have that the sets (curves) of periodic points of period 2 are given by:

$$z_1(\theta)=-\frac{1}{2}+\varepsilon e^{\pi i \theta} \qquad \text{and}\qquad z_2(\theta)=-\frac{1}{2}-\varepsilon e^{\pi i \theta}$$

with 

$$p_{\theta}(z_1(\theta))=z_2(\theta) \qquad \text{and}\qquad  p_{\theta}(z_2(\theta))=z_1(\theta).$$

In other words, the dynamics (in each iteration) ``jumps'' between the two curves $z_1$ and $z_2$.

Similar to the previous case, the curve $\tilde{\gamma}:\TT\rightarrow \CC$ given by 

$$\tilde{\gamma}_1(\theta)=-\frac{1}{2}+\epsilon e^{2\pi i \theta}$$
induces the $2$-curve $\Gamma=(z_1,z_2)$. Furthermore $\tilde{\gamma}_1$ is an invariant curve for the static quadratic polynomial

$$(\theta,z)\mapsto (\theta, z^2+C(\langle 2\theta \rangle)).$$

Now, for sufficiently small $\alpha > 0$ we take the (linear) Lagrange Interpolation polynomial that sends $z_1(\theta)$ to $z_1(\theta + \frac{\alpha}{2})$ and $z_2(\theta)$ to $z_2(\theta + \frac{\alpha}{2})$, we have the following analogous result for $\tau=1$.

\begin{lemma}
 For sufficiently small $\alpha >0$ and $\tau =1$ define the fibred quadratic polynomial.
\begin{displaymath}
    \begin{array}{rccl}
        \widetilde{P}: & \TT\times\CC  & \to &  \TT\times\CC \\
           & (\theta, z)  & \mapsto & (\theta + \frac{\alpha+\tau}{2},z^2+C(\langle 2\theta\rangle)),
    \end{array}
\end{displaymath}
Then the (unfolding) curve $\tilde{\gamma}_1$ is invariant for $\widetilde{P}.$
\end{lemma}

\begin{proposition}
    The $2$-curve $\Gamma = (z_1\ z_2)$ is an invariant multi-curve of \emph{fibred quadratic polynomial} 
    \begin{equation}
        F(\theta, z) = (\theta + \alpha + \tau, \Tilde{\ell}_{\theta/2}(z^2+C(\langle 2\theta\rangle)),
    \end{equation}
    where $\tau=1$ is the jumping integer for the \emph{fqp}.  
\end{proposition}

\begin{remark}
    Conditions (\ref{4.a}), also imply that the parametric curve of the perturbed fibred quadratic polynomial $P(\theta,z)$ has topological degree one with respect to $c_0=1/4$. 
\end{remark}

\section{Searching for 3-curves}
There are two well-known parametrizations for quadratic dynamics. 
\[
z\mapsto P_c(z)=z^2 + c,\ c\in\CC \qquad\text{and}
\qquad z\mapsto Q_\lambda (z) = \lambda z + z^2, \ \lambda\in\CC.
\]
The former generates the picture of the famous Mandelbrot set, defined as 
\[
\mathcal{M}_c = \{c\in\CC : \{P_c^n(c)\}_{n\in\NN} \text{ is bounded}\}
\]
\begin{figure}[ht]
\includegraphics[width=6cm]{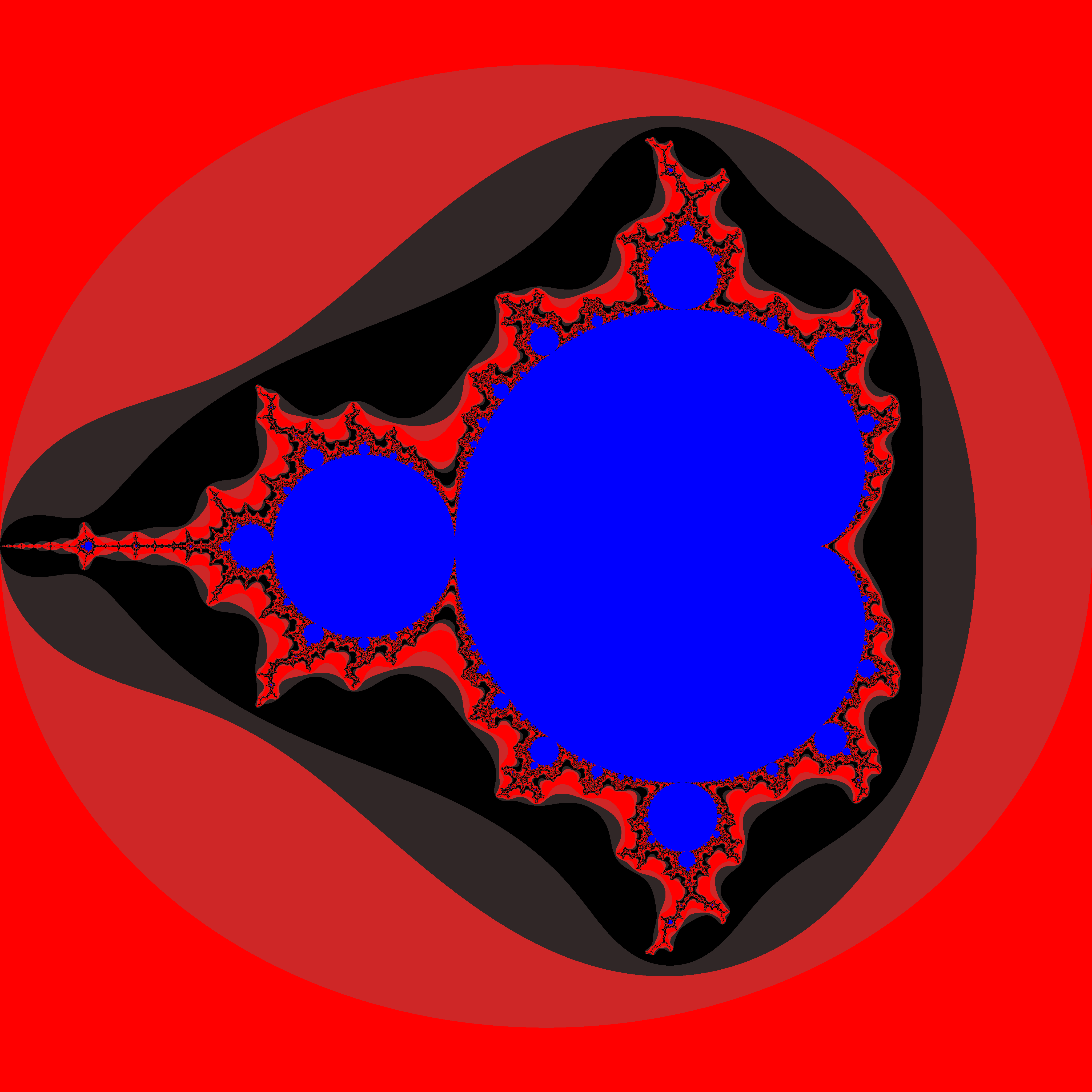}
\centering
\caption{Mandelbrot Set}
\end{figure}
This parametrization is based on the behavior of the only critical value $z_c=c$. On the other hand, the further parametrization is based on the dynamical nature of the two fixed points of the system, In particular, $\lambda\in\DD$ corresponds to a quadratic dynamics with an attracting fixed point (for $\lambda = 0$, $z=0$ is a super-attracting fixed point). The parameter space can be defined as 
\[
\Lambda_\lambda = \{\lambda\in\CC : \{Q_\lambda^n(-\lambda/2)\}_{n\in\NN}\} \text{ is bounded}\}.
\]
\begin{figure}[ht]
\includegraphics[width=8cm]{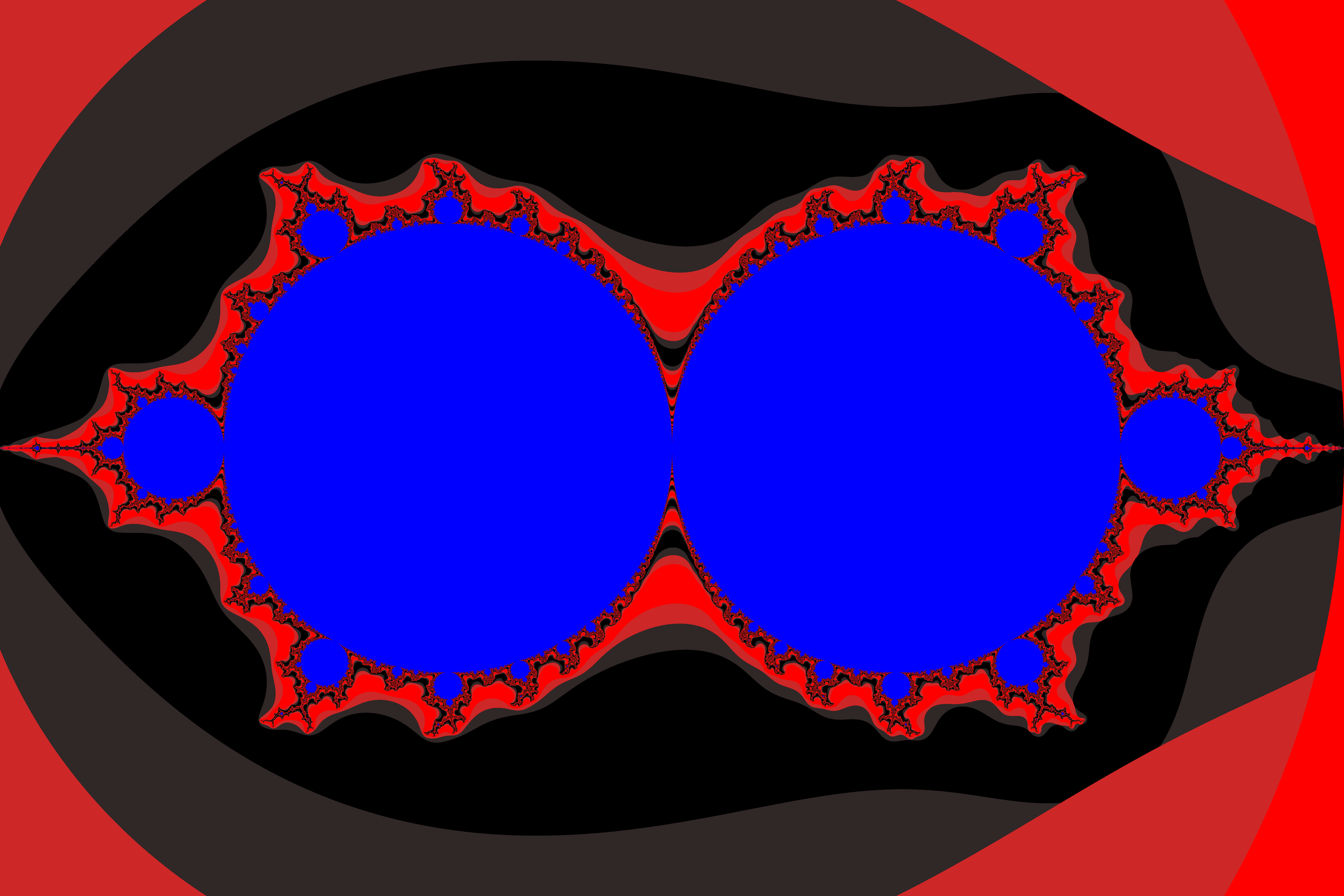}
\centering
\caption{The Lambda space $\Lambda$}
\end{figure}

From the definition, it follows that 
\[
\mathcal{P}ar := \{\lambda = e^{2\pi i\theta} : \theta\in \QQ\},
\]
is the set of parameters with a parabolic fixed point. We have a natural correspondence (2 to 1) between the sets $\Lambda_\lambda$ and $\mathcal{M}_c$, given by the conjugation by $T_{\lambda/2}(z) = z+\lambda/2$ ($z_\lambda=-\lambda/2$ corresponds to the critical point of $Q_\lambda$). The corresponding quadratic function is 
\[
P_\lambda(z) = z^2 + \frac{\lambda}{2}\left(1 - \frac{\lambda}{2}\right). 
\]
Note that $\lambda = 1$ corresponds to the critical value $c=1/4$ and $\lambda=-1$ to $c=-3/4$ as expected. It follows that the map 
\[
\lambda\mapsto\frac{\lambda}{2}\left(1 - \frac{\lambda}{2}\right),
\]
is a correspondence (2-1) between the Lambda space and the Mandelbrot set. 

One of the parabolic fixed points with $3$-petals, in the Lambda space, is given by the parameter $\lambda_0 = e^{\frac{2\pi i}{3}}$, then the quadratic polynomial 
\[
p_{\lambda_0}(z) = z^2 + \lambda_0
\]
\begin{figure}[ht]
\includegraphics[width=6cm]{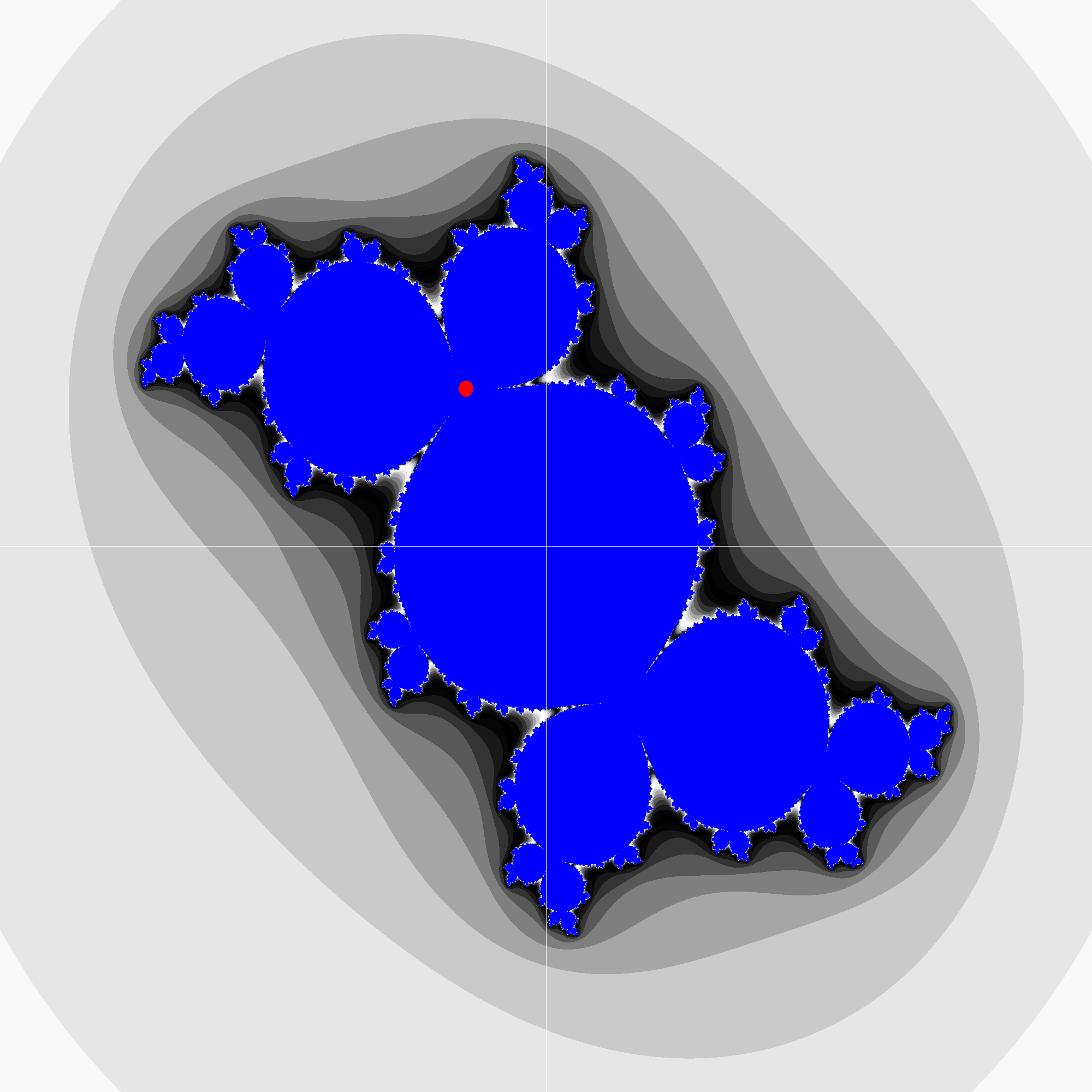}
\centering
\caption{Filled-in Julia set for $p_{\lambda_0}$. The \textbf{red} point corresponds to the parabolic fixed point. It can be appreciate the 3-petals around it. }
\end{figure}
has a parabolic fixed point with 3-petals. So, $\lambda_0$ is a center candidate for the fibred (family of) quadratic polynomial with $\alpha=0$. 

For $\alpha=0$, consider the (static) fibred quadratic polynomial 
\begin{displaymath}
    \begin{array}{rccl}
        P: & \TT\times\CC  & \to &  \TT\times\CC \\
           & (\theta, z)  & \mapsto & (z, p_{\lambda_0}(z) + \varepsilon^2 e^{2\pi i\theta}),
    \end{array}
\end{displaymath}
with $\varepsilon>0$ sufficiently small ($\varepsilon\sim 1/100$). Here are some images of the \emph{filled Julia set,} corresponding to some of the $\theta$'s values. 

%\begin{table}\label{tab_nh}
%    \centering
%    \begin{tabular}{ccc}
%        \toprule
%        a & b & c  \\
%        \midrule
%        \includegraphics[width=4.5cm]{images_3/fib_t=0.0.png} & \includegraphics[width=4.5cm]{images_3/fib_t=0.1.png} & \includegraphics[width=4.5cm]{images_3/fib_t=0.2.png} \\
%        \includegraphics[width=4.5cm]{images_3/fib_t=0.3.png} & \includegraphics[width=4.5cm]{images_3/fib_t=0.4.png} & \includegraphics[width=4.5cm]{images_3/fib_t=0.5.png} \\
%        \includegraphics[width=4.5cm]{images_3/fib_t=0.6.png} & \includegraphics[width=4.5cm]{images_3/fib_t=0.7.png} & \includegraphics[width=4.5cm]{images_3/fib_t=0.8.png} \\
%        \bottomrule
%        \end{tabular}
%        \caption{The table shows the Filled-in Julia sets (in \textbf{blue} for the static polynomial $(\theta,z)\mapsto(\theta,z^2+C(\theta))$, where $C(\theta)$ is a simple loop around the parabolic parameter $\lambda_0$. In all of the images the \textbf{red} point corresponds to the parabolic fixed point of multiplicity 28 of $z\mapsto z^2+c_0$. From left to right and from upper to lower, we have the fiber for the values of $\theta=0.0$ to $\theta=0.8$. }
%        \label{tbl:table_of_figures0}
%\end{table}

\begin{conjecture}\label{conj_1}
By ``tracking'' the corresponding 3-periodic points generated by the perturbation of the parabolic fixed point at $\lambda_0$, we have an invariant 3-curve. 
\end{conjecture}
%\newpage
\subsection{The rational quadratic 3-curve} 
Conjecture \ref{conj_1} in the previous section has two basic obstacles: there is no closed formula to find the 3-period points of the quadratic polynomial, and if we could know the 3-cycle, the Lagrange Interpolation polynomial is no longer \emph{linear}, now is quadratic, so the polynomial obtained with it is now of degree 4. 

%The idea of the conjecture, is that, for $\alpha>0$ small enough, the polynomial obtained after Lagrange interpolation is a \emph{quadratic-like} polynomial, so it can be straightened (as in the classical case) to a quadratic one. 

Nevertheless, it is possible to maintain the degree when we construct the 3-curve, but there is a price to pay. The fibred dynamics is now {\em rational}. It is well known that the Möbius transformations are 3-{\em transitive}: that is, given two set of points $(z_1,z_2,z_3)$ and $(w_1,w_2.w_3)$, there exists a Möbius transformation mapping $z_i$ to $w_i$. 

Let $\gamma:\TT\to\CC^*$ be a simple (and small) loop around the parabolic parameter $\lambda_0$ mentioned before. Consider the ``fibred'' quadratic polynomial given by 
\begin{displaymath}
    \begin{array}{rccl}
        P: & \TT\times\CC  & \to &  \TT\times\CC \\
           & (\theta, z)  & \mapsto & (z, z^2+\gamma(\theta)).
    \end{array}
\end{displaymath}
The polynomial $p_\theta(z)=z^2+\gamma(\theta)$ may be considered as a loop-perturbation of the polynomial $z\mapsto z^2+\lambda_0$ (\emph{parabolic implosion}). In this sense, for every $\theta\in\TT$, $p_\theta$ has a 3-cycle $(\gamma_0(\theta),\gamma_1(\theta),\gamma_2(\theta))$. By the continuity of $\gamma$, this 3-cycle moves continuously on $\TT\times\CC$. 

Now, for $\alpha>0$ sufficiently small, consider the two 3-tuples $(\gamma_0(\theta),\gamma_1(\theta),\gamma_2(\theta))$ and $(\gamma_0(\theta+\alpha),\gamma_1(\theta+\alpha),\gamma_2(\theta+\alpha))$, for each $\theta\in\TT$. Now, for each $\theta\in\TT$, let $M:\widehat{\CC}\to\widehat{\CC}$ be the Möbius transformation that maps $(\gamma_0(\theta),\gamma_1(\theta),\gamma_2(\theta))$ into $(\gamma_0(\theta+\alpha),\gamma_1(\theta+\alpha),\gamma_2(\theta+\alpha))$ (here, $\widehat{\CC}$ denotes the Riemann sphere). The following result follows from a direct calculation. 
\begin{proposition}
    Given the fibred rational quadratic dynamics, 
    \begin{displaymath}
    \begin{array}{rccl}
        Q: & \TT\times\CC  & \to &  \TT\times\CC \\
           & (\theta, z)  & \mapsto & (z+\alpha, M(z^2+\gamma(\theta))),
    \end{array}
\end{displaymath}
for a suitable loop $\gamma$ with index number 1 with respect to the parabolic parameter $\lambda_0$, then the 3-cycle $(\gamma_0(\theta),\gamma_1(\theta),\gamma_2(\theta))$ is an invariant 3-curve for $Q$. 
\end{proposition}

\subsection*{Igsyl Dom\'inguez}
   \noindent Universidad Católica Silva Henríquez, Escuela de Ciencias y Tecnología Educativa,\\
    General Jofré 462, Santiago, Chile. \\
   idominguezc@ucsh.cl.

\end{document}